\documentclass[12pt,reqno]{amsart}
\usepackage{eurosym}
\usepackage{amsfonts}
\usepackage{amsmath}
\usepackage{tikz}
\usetikzlibrary{positioning,calc}
\usepackage{pgfplots}
\usepackage{float}
\usepackage{booktabs}
\usepackage{amssymb,latexsym}
\usepackage{enumerate}
\usepackage[bookmarksnumbered, colorlinks, plainpages]{hyperref}
\usepackage{amsbsy}
\usepackage{amscd}
\usepackage{graphicx}
\usepackage{epsfig}

\usepgfplotslibrary{colormaps}
\pgfplotsset{compat=1.18}
\theoremstyle{plain}
\newtheorem{theorem}{Theorem}[section]

\newtheorem{corollary}{Corollary}[section]

\theoremstyle{definition}
\newtheorem{definition}{Definition}[section]

\newtheorem{example}{Example}[section]

\theoremstyle{remark}
\newtheorem{remark}{Remark}

\numberwithin{equation}{section}

\providecommand{\Authorcontributions}[2]{\section*{Author contributions}#2}
\providecommand{\Competinginterests}[2]{\section*{Competing interests}#2}

\providecommand{\Availability}[2]{\section*{Data availability}#2}

\begin{document}
\date{}
\title[A Matrix-Based Polyalphabetic Algorithm]{A Matrix-Based Polyalphabetic Algorithm for Information Encoding and Decoding Using Number Sequences}
\author[M. Karag\"{o}z]{Muhammet KARAG\"{O}Z}
\address{\.{I}zmir Democracy University, Department of Mathematics\\
35140 Karaba\u{g}lar, \.{I}zmir, T\"{u}rkiye}
\email{muhammet.karagoz@idu.edu.tr}
\author[N. \"{O}zg\"{u}r]{N\.{I}HAL \"{O}ZG\"{U}R}
\address{\.{I}zmir Democracy University, Department of Mathematics\\
35140 Karaba\u{g}lar, \.{I}zmir, T\"{u}rkiye}
\email{nihal.ozgur@idu.edu.tr}
\date{}
\subjclass[2020]{68P30, 11B37, 11B39, 11Y55, 11B83}
\keywords{Polyalphabetic cipher, Leonardo number, Lucas number, Fibonacci number, Jacobsthal number, Friedman test, avalanche effect}

\begin{abstract}

In this paper, we propose a matrix-based polyalphabetic data encoding and
decoding scheme using Fibonacci, Leonardo, Jacobsthal, and Lucas sequences.
The method employs three sequence-based alphabets for character substitution
and a Lucas-based auxiliary alphabet for word separators. A position-dependent
selector,
\[
\sigma=\bigl(v^2+(i-1)+(j-1)\bigr)\pmod 3,
\]
distributes repeated plaintext symbols among different numerical alphabets,
thereby reducing frequency concentration. The resulting numerical matrix is
divided into $3\times 3$ blocks and transformed using powers of the Leonardo
$Q$-matrix with block-dependent keys generated from pre-shared parameters
$(s,p)$. A collision-free public prime $P$ is used to keep ciphertext entries
bounded while preserving unique decoding. A worked example and preliminary
statistical, entropy, avalanche, and timing results indicate that the
proposed modular construction is computationally efficient and provides
improved distributional behavior compared with standard monoalphabetic
substitution.
\end{abstract}

\maketitle

\section{Introduction and preliminaries}
\label{sec:1}

Special integer sequences (e.g.\ Fibonacci sequence, Lucas sequence, Leonardo
sequence, Jacobsthal sequence) have quite widespread application areas such as coding theory.
Their use in encryption is highly important and practical. For
example, in \cite{Stakhov 2006}, an encryption and decryption algorithm was
developed by generalizing Cassini's formula of the Fibonacci sequence.
In \cite{Tas et al 2018}, an encryption and decryption algorithm has been developed by
dividing messages into block matrices using the Fibonacci $Q$-matrix. A
recent example was conducted in \cite{Ozcevik et al 2023} using the Leonardo
sequence and the Leonardo $Q$-matrix. The applications of Jacobsthal and Jacobsthal-Lucas numbers in coding theory have been studied in \cite{Kuloglu}. For further studies, see, for example, \cite{Aljalali et al,Madak,Ozgur et al,Ozyilmaz,Rawal}, and the references therein.

In this study, a novel polyalphabetic encryption scheme is proposed, which
leverages four numerical sequences---Fibonacci, Lucas, Leonardo, and Jacobsthal---along
with four distinct alphabets. Polyalphabetic encryption uses multiple substitution alphabets, unlike monoalphabetic methods,
which rely on a single substitution table. This ensures that the
same plaintext character can be encrypted as different ciphertext characters
depending on its position and occurrence count. One of the most prominent examples
of polyalphabetic encryption is the Vigen\`{e}re cipher. For further details, see, for example, \cite{Ahmed}, \cite{Aung et al}, \cite{Hannan}, \cite{Noman}, and the references therein. Although the proposed method is primarily polyalphabetic in structure, it also exhibits a homophonic effect. In a polyalphabetic substitution scheme, the substitution alphabet changes according to a position-, key-, or rule-dependent mechanism. In contrast, a homophonic substitution scheme assigns several possible ciphertext representatives to the same plaintext symbol in order to reduce frequency concentration. In the present construction, the alphabet selector determines whether a plaintext character is represented by a Leonardo, Fibonacci, or Jacobsthal value. Hence, the scheme is best described as a polyalphabetic substitution scheme with homophonic behavior.

Polyalphabetic encryption offers enhanced
security over monoalphabetic techniques. By encrypting the same letter
differently each time, it significantly mitigates the risk of frequency
analysis~\cite{Gaines1956}. When a sufficiently long and complex key is employed, unauthorized decryption through classical cryptanalysis becomes significantly more challenging \cite{Hannan}. Despite these strengths, its security heavily depends on
the secrecy of the key; if the key is compromised, the entire encryption
system becomes vulnerable \cite{Alrammahi}.

We begin by recalling some basic properties of the Fibonacci, Lucas, Leonardo, and Jacobsthal
sequences. The Fibonacci sequence $\left( F_{k}\right) $ (the sequence A000045 of the
on-line encyclopedia of integer sequences (OEIS) \cite{Sloane}) is defined by the recurrence relation:%
\begin{equation}
F_{k+1}=F_{k}+F_{k-1},\left( k\geq 1\right)   \label{FibonacciS}
\end{equation}%
with the initial terms $F_{0}=0$ and $F_{1}=1$ (see \cite{Koshy}). The Lucas sequence $\left(
L_{k}\right) $ (A000032 in the OEIS) is defined by
\begin{equation}
L_{k+1}=L_{k}+L_{k-1},\left( k\geq 1\right)   \label{LucasS}
\end{equation}%
with $L_{0}=2$ and $L_{1}=1$ (see \cite{Koshy}). The Leonardo sequence $\left( \mathcal{L}_{k}\right) $ (A001595 in the OEIS) is defined by
\begin{equation}
\mathcal{L}_{k+1}=\mathcal{L}_{k}+\mathcal{L}_{k-1}+1,\quad k\geq 1,
\label{LeonardoS}
\end{equation}%
with $\mathcal{L}_{0}=\mathcal{L}_{1}=1$ (see \cite{Catarino Borges 2020}). The Jacobsthal sequence $\left( J_{m}\right)$ (A001045 in the OEIS) is defined by
\begin{equation}
J_{m+1} = J_{m} + 2 J_{m-1}, \quad (m \geq 1) \label{JacobsthalS}
\end{equation}
with $J_{0}=0$ and $J_{1}=1$, (see \cite{Sloanebook,Bensella2024,Kuloglu}).

We will use the following identities:%
\begin{equation}
L_{k}=F_{k+1}+F_{k-1},\left( k\geq 0\right) ,  \label{LF-1}
\end{equation}%
\begin{equation}
\mathcal{L}_{k}=\frac{2}{5}\left( L_{k}+L_{k+2}\right) -1,\left( k\geq
0\right) ,  \label{Le-L-1}
\end{equation}%
\begin{equation}
L_{k}=\frac{\mathcal{L}_{k-2}+\mathcal{L}_{k}+2}{2},  \label{Le-L-2}
\end{equation}%
and
\begin{equation}
\mathcal{L}_{n} = 2F_{n+1} - 1, \quad (n \geq 0) \label{Le-F-2}
\end{equation}
(see \cite{Alp Kocer, Catarino Borges 2020, Koshy, Soykan}, and the references therein).

There is a practical method to obtain the elements of these sequences using
matrices. In \cite{Alp
Kocer}, the $Q$ matrix associated with Leonardo numbers was defined as%
\begin{equation*}
Q=\left(
\begin{array}{ccc}
2 & 1 & 0 \\
0 & 0 & 1 \\
-1 & 0 & 0%
\end{array}%
\right)
\end{equation*}%
with $\det (Q)=-1$, and
\begin{equation*}
Q^{n}=\frac{1}{2}\left(
\begin{array}{ccc}
\mathcal{L}_{n+2}-1 & \mathcal{L}_{n+1}-1 & \mathcal{L}_{n}-1 \\
1-\mathcal{L}_{n} & 1-\mathcal{L}_{n-1} & 1-\mathcal{L}_{n-2} \\
1-\mathcal{L}_{n+1} & 1-\mathcal{L}_{n} & 1-\mathcal{L}_{n-1}%
\end{array}%
\right) .
\end{equation*}%
The Cassini identity for Leonardo numbers is%
\begin{equation}
\left( \mathcal{L}_{n}\right) ^{2}-\mathcal{L}_{n-1}\mathcal{L}_{n+1}=%
\mathcal{L}_{n-1}-\mathcal{L}_{n-2}+4(-1)^{n}  \label{Cassini identity}
\end{equation}%
(see \cite{Alp Kocer} and \cite{Catarino Borges 2020}).

The organization of this paper is as follows.
Section~\ref{sec:2} presents the construction of the proposed
polyalphabetic data encoding scheme based on the four sequences and the Leonardo
\(Q\)-matrix, together with the modular arithmetic framework that ensures bounded
ciphertext values.
Section~\ref{sec:3} presents the statistical evaluation and computational performance benchmarks.
Section~\ref{sec:4} discusses the strengths and limitations
of the proposed scheme.
Section~\ref{sec:5} gives the concluding remarks.

\section{A new polyalphabetic data encoding scheme via number sequences}
\label{sec:2}

The aim of this section is to present a new polyalphabetic encryption strategy using four number sequences: Fibonacci,
Lucas, Leonardo, and Jacobsthal. To achieve this,
we need to identify the common terms among these sequences.
From \cite{Koshy}, the common terms of the Fibonacci and Lucas sequences are
\begin{equation*}
F_{1}=F_{2}=L_{1}=1,\quad F_{3}=L_{0}=2,\quad F_{4}=L_{2}=3.
\end{equation*}%
From \cite{Tripathi}, the common terms of the Fibonacci and Leonardo sequences are
\begin{equation*}
\mathcal{L}_{0}=\mathcal{L}_{1}=F_{1}=F_{2}=1,\quad\mathcal{L}_{2}=F_{4}=3,\quad%
\mathcal{L}_{3}=F_{5}=5.
\end{equation*}%

In the following theorem, we give a short proof of this fact (see also Theorem 1 in \cite{Tripathi}).

\begin{theorem}
\label{thm:21} If $k>3,n>5$ are integers, then $\mathcal{L}_{k}\neq F_{n}$.
\end{theorem}

\begin{proof}
Since
\[
\mathcal L_k=2F_{k+1}-1,
\]
for $k\geq 4$ we have
\[
F_{k+2}<2F_{k+1}-1<F_{k+3}.
\]
Indeed,
\[
2F_{k+1}-1-F_{k+2}=F_{k+1}-F_k-1=F_{k-1}-1>0
\]
for $k\geq 4$, and
\[
F_{k+3}-(2F_{k+1}-1)=F_{k+2}+F_{k+1}-2F_{k+1}+1=F_{k+2}-F_{k+1}+1=F_k+1>0.
\]
Hence
\[
F_{k+2}<\mathcal L_k<F_{k+3}.
\]
Therefore $\mathcal L_k$ lies strictly between two consecutive Fibonacci numbers and cannot itself be a Fibonacci number.
\end{proof}

In the following theorem we prove that the Lucas and Leonardo sequences have only finitely many common terms. The common terms of the Lucas and Leonardo sequences are
\begin{equation*}
\mathcal{L}_{0}=\mathcal{L}_{1}=L_{1}=1,\quad\mathcal{L}_{2}=L_{2}=3.
\end{equation*}

\begin{theorem}
\label{thm:22} If $k,m>2$ are integers, then $\mathcal{L}_{k}\neq L_{m}$.
\end{theorem}

\begin{proof}
For $k\geq 3$, we obtain
\[
L_k<\mathcal L_k<L_{k+1}.
\]
Indeed, using \eqref{LF-1} and \eqref{Le-F-2} we have
\[
\mathcal L_k-L_k=(2F_{k+1}-1)-(F_{k+1}+F_{k-1})=F_k-1>0,
\]
and
\[
L_{k+1}-\mathcal L_k=(F_{k+2}+F_k)-(2F_{k+1}-1)=F_{k-2}+1>0.
\]
Therefore each Leonardo number $\mathcal L_k$ with $k\geq 3$ lies strictly between two consecutive Lucas numbers. Hence no Lucas term with index greater than $2$ equals a Leonardo term.
\end{proof}
For the Leonardo--Jacobsthal intersection, Bensella and Behloul \cite{Bensella2024} established the following theorem.
\begin{theorem}
\label{thm:LeoJac}
All solutions of $\mathcal{L}_{n}=J_{m}$ in non-negative integers with
$n\geq 2$ and $m\geq 2$ are
\[
(n,m)\in\{(2,3),(3,4)\},
\]
corresponding to
\[
\mathcal{L}_{2}=J_{3}=3
\qquad\text{and}\qquad
\mathcal{L}_{3}=J_{4}=5.
\]
\end{theorem}

\begin{theorem}
\label{thm:FibJac}
The common terms of the Fibonacci and Jacobsthal sequences are
\[
\begin{aligned}
F_{0}&=J_{0}=0,\\
F_{1}&=F_{2}=J_{1}=J_{2}=1,\\
F_{4}&=J_{3}=3,\qquad
F_{5}=J_{4}=5,\qquad
F_{8}=J_{6}=21.
\end{aligned}
\]
\end{theorem}

\begin{proof}
The equality $F_0 = J_0 = 0$ is clear, and $F_1 = F_2 = J_1 = J_2 = 1$
follows directly from the definitions. It remains to find all positive
common terms with $r \geq 3$.

Since $J_1 = J_2 = 1$, the equation $F_k = J_r$ embeds into the
product equation $F_k = J_m J_n$ by setting $(m,n) \in
\{(r,1),(r,2),(1,r),(2,r)\}$. The complete solution of $F_k = J_m J_n$
in positive integers, obtained in \cite{ErduvanKeskin2021}, yields the
triples having at least one index in $\{1,2\}$:
\[
\begin{aligned}
(k,m,n) \in \{&(1,1,1),(2,1,1),(1,1,2),(2,1,2),(1,2,2),(2,2,2),\\
               &(4,1,3),(4,2,3),(5,1,4),(5,2,4),(8,1,6),(8,2,6)\}.
\end{aligned}
\]
These give precisely $F_4 = J_3 = 3$, $F_5 = J_4 = 5$, and
$F_8 = J_6 = 21$. Since \cite{ErduvanKeskin2021} provides the
complete solution, no further common terms exist.
\end{proof}
\begin{theorem}
\label{thm:LucJac}
The common terms of the Lucas and Jacobsthal sequences are
\[
L_{1}=J_{1}=J_{2}=1,\qquad
L_{2}=J_{3}=3,\qquad
L_{5}=J_{5}=11.
\]
\end{theorem}

\begin{proof}
Since $J_0 = 0$, any common term $L_k = J_r$ can be written as
$J_r + J_0 = L_k$, embedding it into the equation $J_n + J_m = L_k$.
The complete solution of this equation in non-negative integers with
$n \geq m$, obtained in \cite{Salah2024}, is
\[
\begin{aligned}
\{&(1,1,0),(2,1,0),(2,2,0),(1,0,1),(2,0,1),\\
  &(3,0,2),(3,1,3),(3,2,3),(5,0,5)\}.
\end{aligned}
\]
The triples with $m = 0$ are $(1,0,1)$, $(2,0,1)$, $(3,0,2)$, $(5,0,5)$,
giving
\[
J_1 = J_2 = L_1 = 1, \qquad J_3 = L_2 = 3, \qquad J_5 = L_5 = 11.
\]
Since \cite{Salah2024} provides the complete solution, no further common
terms exist.
\end{proof}
\begin{corollary}
\label{cor:21}
Let $k\geq 5$, $m\geq 6$, $n\geq 5$, and $r\geq 7$ be integers. Then the four
values
\[
\mathcal{L}_{k},\qquad L_{m},\qquad F_{n},\qquad J_{r}
\]
are pairwise distinct; that is,
\begin{equation}
\mathcal{L}_{k}\neq L_{m}, \quad
L_{m}\neq F_{n}, \quad
\mathcal{L}_{k}\neq F_{n}, \quad
\mathcal{L}_{k}\neq J_{r}, \quad
L_{m}\neq J_{r}, \quad
F_{n}\neq J_{r}.
\label{eqn:24}
\end{equation}
\end{corollary}

This corollary provides the theoretical foundation for constructing a
new polyalphabetic data encoding scheme in which the Leonardo, Fibonacci,
and Jacobsthal sequences generate three pairwise disjoint character
alphabets, while Lucas numbers are used solely to encode word separators.

\subsection{Construction of a new polyalphabetic data encoding scheme}

We construct a polyalphabetic encryption procedure utilizing matrices and four number sequences: Leonardo, Fibonacci, Lucas, and Jacobsthal. This scheme is designated as the \emph{FLLJ-POLY Alphabetic Encryption Algorithm}. Accordingly, FLLJ-POLY is regarded as a polyalphabetic scheme with homophonic behavior, since the same plaintext symbol may be represented by different sequence values depending on the alphabet selected by the rule.
\subsubsection{Choice of a collision-free prime modulus}
\label{subsec:collision-free-prime}

Before introducing the modular implementation, we clarify the role of the
prime modulus $\mathbf{P}$. The purpose of $\mathbf{P}$ is not to enlarge the secret key
space, but to ensure that all numerical representatives used in the character
tables remain uniquely identifiable after reduction modulo $\mathbf{P}$. In other
words, if two distinct sequence values are used to represent two different
symbols before modular reduction, then they must not become equal after
reduction modulo $\mathbf{P}$. Thus, for the fixed alphabet set used in the algorithm,
we require the reduction map modulo $\mathbf{P}$ to be injective. Once such a prime is
chosen, each residue appearing in the ciphertext alphabet corresponds to a
unique original sequence value, and hence to a unique plaintext symbol or
separator.

The modular reduction prime $\mathbf{P}$ must therefore be chosen so that the
numerical representatives used in the character tables remain distinguishable
after reduction. Let
\[
\mathcal{C}=\{x_1,x_2,\ldots,x_N\}
\]
be the finite set consisting of all Fibonacci, Leonardo, Lucas and Jacobsthal
values used in the encryption process before reduction modulo $\mathbf{P}$;
specifically, the Fibonacci table starts from \(F_5\), the Leonardo table
from \(\mathcal{L}_5\), and the Jacobsthal table from \(J_7\), in accordance
with the index ranges of Corollary~\ref{cor:21}. The Lucas values used for
word separators are also included in \(\mathcal{C}\). By Corollary~\ref{cor:21}, the prescribed index ranges for the three principal
character tables are pairwise disjoint before modular reduction. The Lucas
separator values, which depend on \(a_0\), are included in the finite set
\(\mathcal C\) for the chosen key instance and are checked together with the
three principal tables in the collision-free condition. However, distinct
integers may become congruent modulo a prime \(\mathbf{P}\). Hence, we must
exclude primes that create such modular collisions.

Define
\[
\Delta \;:=\; \prod_{1\leq i<j\leq N}(x_i-x_j).
\]
Since the elements of \(\mathcal{C}\) are pairwise distinct, we have
\(\Delta\neq 0\). For any prime $\mathbf{P}$, the reduction map
\[
\rho_{P}:\mathcal{C}\longrightarrow \mathbb{Z}_{P},
\qquad
x\longmapsto u \pmod{\mathbf{P}},
\]
is injective if and only if
\[
x_i \not\equiv x_j \pmod{\mathbf{P}}
\qquad
(1\leq i<j\leq N).
\]
Equivalently, this is the same as requiring
\[
\mathbf{P}\nmid (x_i-x_j)
\qquad
(1\leq i<j\leq N),
\]
or, equivalently,
\[
\mathbf{P}\nmid \Delta.
\]
Thus, the condition \(\mathbf{P}\nmid \Delta\) guarantees that no two distinct elements
of \(\mathcal{C}\) collapse to the same residue class modulo \(\mathbf{P}\).
\begin{definition}
A prime \(P\) is said to be \emph{collision-free} for the set
\(\mathcal{C}\) if
\[
\Delta \not\equiv 0 \pmod{\mathbf{P}},
\]
that is, if $\mathbf{P}$ does not divide the discriminant-like product
\(\Delta\). Define the finite set of \emph{forbidden primes} by
\[
\mathcal{F}:=\{\,p \text{ prime}: p\mid \Delta\,\}.
\]
A prime $\mathbf{P}$ is admissible if and only if $\mathbf{P}\notin\mathcal{F}$.
\end{definition}

Since \(\Delta\) is a nonzero integer, it has only finitely many prime
divisors. Therefore, the forbidden set \(\mathcal{F}\) is finite. By
Euclid's theorem, there exist infinitely many primes, and consequently
infinitely many admissible collision-free primes.

For such an admissible prime $\mathbf{P}$, the inverse character map is well-defined.
Indeed, since \(\rho_P\) is injective on \(\mathcal{C}\), the inverse map
\[
\rho_P^{-1}:\rho_P(\mathcal{C})\longrightarrow \mathcal{C}
\]
is uniquely determined by
\[
\rho_P^{-1}(u)=x,
\]
where \(x\in\mathcal{C}\) is the unique element satisfying
\[
x \equiv u \pmod{\mathbf{P}}.
\]
This uniqueness is essential in the decoding step: Each modular ciphertext
symbol belonging to the reduced character alphabet can be traced back to one
and only one original sequence value.

In the computations and worked examples of this paper, we select
\(\mathbf{P=10159}\). It is one of the admissible
collision-free primes for the specific alphabet set \(\mathcal{C}\) used in
our implementation. Its 14-bit binary representation
\[
10159=(10011110101111)_2
\]
is also consistent with the fixed-length binary encoding used in the
avalanche analysis of Section~\ref{subsec:avalanche}. Admissibility was
verified computationally by checking
\[
\gcd(\Delta,10159)=1,
\]
which confirms that
\[
10159\notin\mathcal{F}.
\]
Therefore, for the chosen alphabet set, reduction modulo \(10159\) preserves
the uniqueness of all sequence representatives.

We emphasize that \(\mathbf{P=10159}\) is not a secret security parameter. Its role is to ensure injectivity of the character
mapping and to keep ciphertext entries bounded for computational efficiency.
The security of the scheme rests on the secrecy of the pre-shared key pair
\((s,p)\), while \(\bf{P}\) serves as a public implementation modulus. Any other
prime \(\bf{P}\notin\mathcal{F}\) could be used without affecting the
well-definedness of the inverse character mapping.

The Fibonacci, Leonardo, and Jacobsthal character tables are fixed
independently of the message and the key. Since these sequences are infinite,
the tables can be extended to any prescribed finite character set by assigning
distinct sequence values from suitable predetermined index ranges. In the proposed algorithm, the Lucas indices are shifted by
\(a_0\), the first component of the key sequence generated from the
pre-shared parameters \((s,p)\). Thus, if the separator symbol \(*\) appears
for the \(t\)-th time, it is encoded as
\[
L_{z_t}=L_{a_0+t}.
\]
Hence, unlike the three principal character alphabets, the Lucas separator
values vary with the key through \(a_0\), adding further variability to the
encoding of word separators.
For a fixed implementation, one chooses an upper bound \(T_{\max}\) for the
number of word separators supported by the scheme. Then, for each admissible
key instance, the collision-free condition for the public modulus
\(\mathbf{P}\) is verified on the union of the three fixed character tables
and the corresponding Lucas separator set
\[
\{L_{a_0+1},L_{a_0+2},\ldots,L_{a_0+T_{\max}}\}.
\]
\subsubsection{FLLJ-POLY Encoding Algorithm}
\mbox{}\\
{Step 1.} Determine the message text to be encoded.

\noindent{Step 2.}
Convert the message into a square matrix $M$ of size $3m\times 3m$ by inserting
$*$ between consecutive words and, if necessary, appending additional $*$'s at
the end of the message.

\noindent{Step 3.}
The block parameter is
\[
k=m^2,
\]
which equals the number of $3\times 3$ submatrices. The sender and the
receiver pre-share two independent secret parameters: a seed
$s\in\mathbb{Z}^{+}$ and a key-generation prime $p$. The seed may be chosen
from a large interval, for example
\[
1\leq s<2^{128},
\]
while $p$ is chosen as a secret $b_p$-bit prime satisfying
\[
2^{b_p-1}\leq p<2^{b_p}.
\]
In the recommended parameter setting, one may take $b_s=128$ for the seed and
$b_p=128$ for the key-generation prime. Thus, $s$ is selected and $p$ is selected from the set of $128$-bit primes. The parameters \(s\) and \(p\) are chosen so that
\[
a_0=(L_s k \bmod p)+1\ge 5.
\]
If this condition is not satisfied, a new pair \((s,p)\) is selected.
The per-block private key sequence
\[
a_0,a_1,\ldots,a_{k-1}
\]
is generated iteratively by
\begin{align}
a_0 &= \left(\mathcal{L}_s \cdot k\right)(\bmod p) + 1,  \\
a_{i+1} &= \left(\mathcal{L}_s a_i + a_i^2 + i\right)(\bmod p),
\qquad i=0,1,\ldots,k-2.
\end{align}

Equivalently, the seed affects the key sequence through the residue
\[
\lambda_s=\mathcal{L}_s(\bmod p).
\]
Thus, the effective key-generation rule can be written as
\[
a_0=(\lambda_s k)(\bmod p)+1,
\qquad
a_{i+1}=(\lambda_s a_i+a_i^2+i)(\bmod p).
\]

Since the recurrence defining $a_{i+1}$ is quadratic and non-autonomous, the resulting key sequence is not amenable to simple linear reconstruction. Moreover, the public block parameter \(k=m^2\) is determined only by the number of
\(3\times 3\) plaintext blocks and is independent of the secret parameters
\(s\) and \(p\). The values of \(s\) and \(p\) are pre-shared between the sender
and the receiver, in a manner analogous to a symmetric-key setting such as AES (see \cite{DaemenRijmen2002}).
Thus, they are not derived from \(k\) and are not transmitted together with the
ciphertext.

\medskip\noindent{Step 4.}
Construct the Fibonacci, Leonardo, and Jacobsthal character tables. The Fibonacci table starts from $F_5$, the
Leonardo table from $\mathcal{L}_5$, and the Jacobsthal table from $J_7$.
Lucas numbers are used for word separators.

\medskip

\noindent{Step 5.}
The choice of $\mathbf{P}$ follows the criterion in
Subsection~\ref{subsec:collision-free-prime}. We first replace the symbolic
entries of the message matrix $M$ with their corresponding sequence values
according to the following rules:
\begin{itemize}
    \item \textbf{For letters:} If a character at position $(i,j)$,
    where $i,j \in \{1,\ldots,3m\}$ denote the row and column indices in the
    message matrix $M$, appears for the $v$-th time, compute the selector
    \[
    \sigma=\bigl(v^{2}+(i-1)+(j-1)\bigr)\pmod 3,
    \qquad 1\leq i,j\leq 3m.
    \]
    Replace the character with its corresponding value from the Leonardo table
    if $\sigma=0$, from the Fibonacci table if $\sigma=1$, and from the
    Jacobsthal table if $\sigma=2$.

    \item \textbf{For word separators ($*$'s):} Replace the $t$-th occurrence
    of $*$ with the Lucas number
    \[
    L_{z_t},
    \qquad \text{where } z_t=a_0+t.
    \]
\end{itemize}

After all entries have been replaced by their corresponding sequence values,
we reduce the resulting  matrix modulo $\mathbf{P}$. The matrix
obtained in this way is denoted by $\widehat{M}$.

\medskip
\noindent{Step 6.} Divide the converted matrix, denoted by $\widehat{M}$, into
submatrices of size $3\times 3$ from left to right as follows:%
\begin{equation*}
\widehat{M}=\left(
\begin{array}{cccc}
\widehat{M}_{1} & \widehat{M}_{2} & \cdots & \widehat{M}_{m} \\
\widehat{M}_{m+1} & \widehat{M}_{m+2} & \cdots & \widehat{M}_{2m} \\
\vdots & \vdots & \cdots & \vdots \\
\widehat{M}_{m(m-1)+1} & \widehat{M}_{m(m-1)+2} & \cdots & \widehat{M}%
_{m^{2}}%
\end{array}%
\right) .
\end{equation*}
\noindent{Step 7.}
Encrypt each $3\times 3$ block $\widehat{M}_j$ using its corresponding
private key $a_{j-1}$:
\begin{equation}
K_j \equiv \widehat{M}_j Q^{a_{j-1}} \pmod{\mathbf{P}},
\qquad 1 \le j \le k.
\end{equation}
Since $\det(Q)=-1$, we have $\det(Q^{a_{j-1}})=(-1)^{a_{j-1}}\equiv \pm 1\pmod{\mathbf{P}}$
for every $a_{j-1}$. In particular,
$\gcd(\det(Q^{a_{j-1}}),P)=1$, so $Q^{a_{j-1}}$ is invertible
modulo $\bf{P}$ for all choices of $a_{j-1}$.
Then compute the determinant of each ciphertext block:
\begin{equation}
d_j \equiv \det(K_j) \pmod{\mathbf{P}},
\qquad 1 \le j \le k.
\end{equation}

\medskip\noindent{Step 8.} Construct the code matrix $K_{M}$ using the sub-code
matrices $K_{j}$ ($1\leq j\leq m^{2}$) according to the following: The first column of the matrix $K_{M}$ contains the determinants of the
submatrices, respectively. Then, the other elements of any row contain the
elements of the corresponding submatrix written row by row except the hidden
elements. If the cofactor of the designated hidden entry is congruent to zero modulo $\mathbf{P}$, the entry corresponding to $(k+1) \pmod 9$ in the row-wise ordering is selected instead (where a result of $0$ corresponds to the $9$-th entry), and the receiver applies the same fallback rule. If the cofactor of the initially designated hidden entry is congruent to zero modulo \(\textbf{P}\),
the encoder selects the first entry, in row-wise order, whose cofactor is nonzero modulo
\(\textbf{P}\). The index of this selected hidden entry is transmitted together with the block
determinant. If no such entry exists, the full block is transmitted without hiding an
entry. For example, if the elements $K_{21}^{j}$($1\leq j\leq m^{2}$) are
hidden. Then we have
\begin{equation*}
K_{M}=\left(
\begin{array}{ccccccccc}
\det (K_{1}) & K_{11}^{1} & K_{12}^{1} & K_{13}^{1} & K_{22}^{1} & K_{23}^{1}
& K_{31}^{1} & K_{32}^{1} & K_{33}^{1} \\
\det (K_{2}) & K_{11}^{2} & K_{12}^{2} & K_{13}^{2} & K_{22}^{2} & K_{23}^{2}
& K_{31}^{2} & K_{32}^{2} & K_{33}^{2} \\
\vdots & \vdots & \vdots & \vdots & \vdots & \vdots & \vdots & \vdots &
\vdots \\
\det (K_{m^{2}}) & K_{11}^{m^{2}} & K_{12}^{m^{2}} & K_{13}^{m^{2}} &
K_{22}^{m^{2}} & K_{23}^{m^{2}} & K_{31}^{m^{2}} & K_{32}^{m^{2}} &
K_{33}^{m^{2}}%
\end{array}%
\right) _{m^{2}\times 9}.
\end{equation*}

{Step 9. }Send the code matrix $K_{M}$ to the recipient.

\subsubsection{FLLJ-POLY Decoding Algorithm}
\mbox{}\\
\noindent{Step 1.} Using the pre-shared $(s,p)$ and the extracted block parameter $k$, compute the key sequence $a_0, \dots, a_{k-1}$. Calculate the modular inverse matrices $(Q^{a_{j-1}})^{-1} \pmod{\bf{P}}$ using number theory.

\noindent{Step 2.} Reconstruct the blocks $K_j$ by calculating the missing hidden elements using the provided determinants $d_j$ and the known row elements.

\noindent{Step 3.} Recover the plaintext submatrices $\widehat{M}_j$ by decrypting each block:
\begin{equation}
\widehat{M}_j \equiv K_j \times (Q^{a_{j-1}})^{-1} \pmod{\bf{P}}.
\end{equation}
Assemble the full matrix $\widehat{M}$ from these submatrices.

\noindent{Step 4.} Convert the numerical values in $\widehat{M}$ back to text using the
Fibonacci, Leonardo, Jacobsthal, and Lucas tables. Since the reduction map
$\rho_P$ is injective on $\mathcal{C}$, the inverse character mapping is
well-defined.

\noindent{Step 5. }Write the message text.
\noindent{Step 6. }End of the algorithm.
\newpage
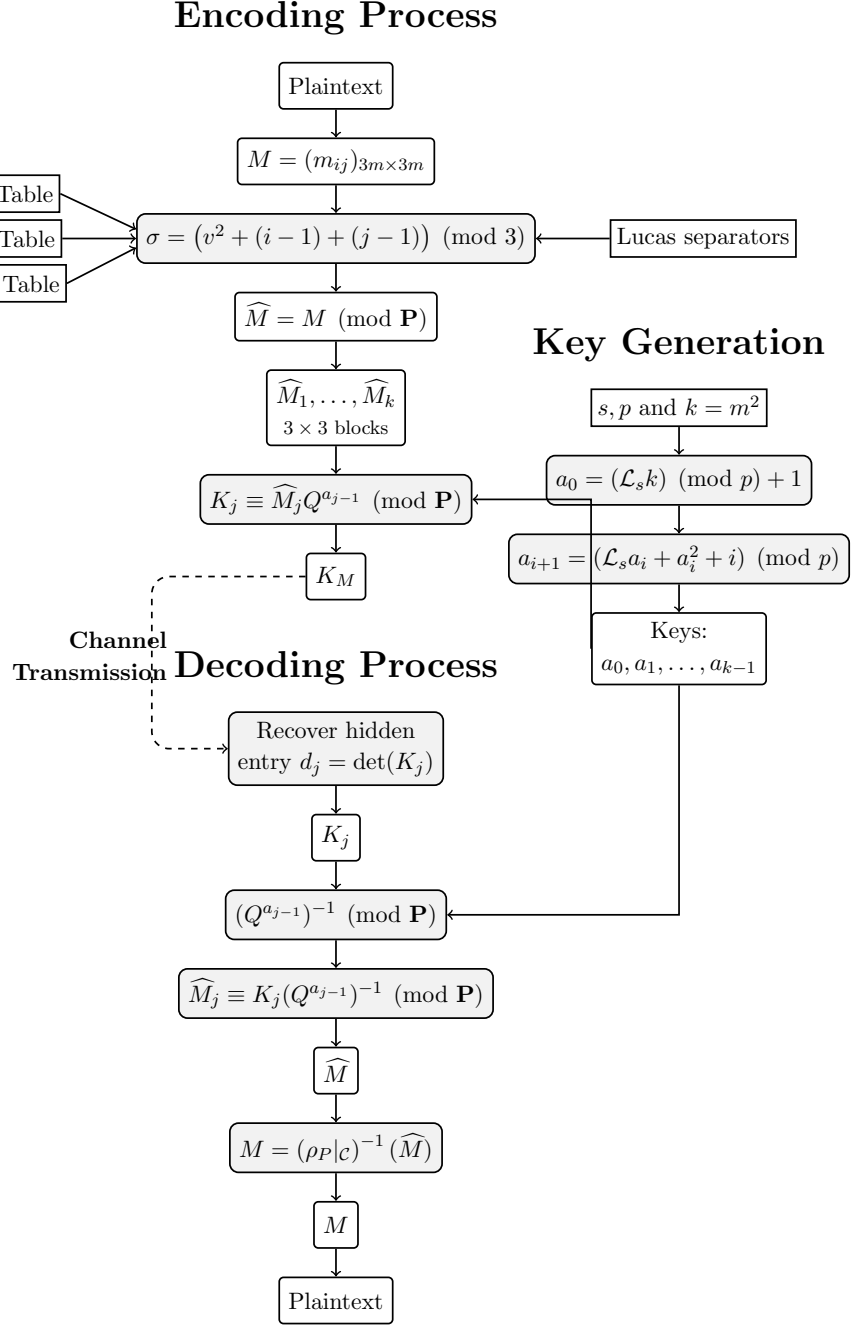
\begin{figure}[H]
\centering
\resizebox{!}{0.80\textheight}{
\begin{tikzpicture}[
    node distance=0.45cm,
    every node/.style={font=\small},
    box/.style={
        rectangle,
        draw=black,
        thick,
        rounded corners=2pt,
        align=center,
        minimum height=0.75cm,
        inner sep=0.15cm
    },
    sbox/.style={
        rectangle,
        draw=black,
        thick,
        align=center,
        minimum height=0.6cm,
        inner sep=0.1cm
    },
    op/.style={
        rectangle,
        draw=black,
        thick,
        fill=gray!10,
        rounded corners=4pt,
        align=center,
        minimum height=0.8cm,
        inner sep=0.15cm
    },
    arr/.style={->, thick}
]

\node (pt) [box] {Plaintext};

\node[font=\Large\bfseries, above=0.3cm of pt] {Encoding Process};

\node (M) [box, below=of pt]
{$M=(m_{ij})_{3m\times 3m}$};

\node (sel) [op, below=of M] {$\sigma=\bigl(v^2+(i-1)+(j-1)\bigr)\pmod {3}$};

\node (tab2) [sbox, left=1.2cm of sel] {Fibonacci Table};
\node (tab1) [sbox, above=0.1cm of tab2] {Leonardo Table};
\node (tab3) [sbox, below=0.1cm of tab2] {Jacobsthal Table};
\node (luc)  [sbox, right=1.2cm of sel] {Lucas separators};

\node (Mh) [box, below=of sel] {$\widehat{M}=M \pmod{\mathbf{P}}$};

\node (blk) [box, below=of Mh] {$\widehat{M}_1,\ldots,\widehat{M}_k$\\[0.5mm] \scriptsize $3\times 3$ blocks};

\node (enc) [op, below=of blk] {$K_j \equiv \widehat{M}_jQ^{a_{j-1}}\pmod{\mathbf{P}}$};

\node (KM) [box, below=of enc] {$K_M$};

\draw[arr] (pt) -- (M);
\draw[arr] (M) -- (sel);
\draw[arr] (sel) -- (Mh);
\draw[arr] (Mh) -- (blk);
\draw[arr] (blk) -- (enc);
\draw[arr] (enc) -- (KM);

\draw[arr] (tab1.east) -- ([yshift=1.5mm]sel.west);
\draw[arr] (tab2.east) -- (sel.west);
\draw[arr] (tab3.east) -- ([yshift=-1.5mm]sel.west);
\draw[arr] (luc.west) -- (sel.east);

\node (seed) [sbox, right=3.0cm of blk] {$s, p \text{ and } k=m^2$};

\node[font=\Large\bfseries, above=0.3cm of seed] {Key Generation};

\node (a0) [op, below=of seed] {$a_0=(\mathcal{L}_s k)\pmod{p}+1$};
\node (ai) [op, below=of a0] {$a_{i+1} = (\mathcal{L}_s a_i + a_i^2 + i) \pmod{p}$};
\node (keys) [box, below=of ai] {Keys:\\[0.5mm] $a_0,a_1,\ldots,a_{k-1}$};

\draw[arr] (seed) -- (a0);
\draw[arr] (a0) -- (ai);
\draw[arr] (ai) -- (keys);

\draw[arr] (keys.west) |- (enc.east);

\node (rec) [op, below=1.8cm of KM] {Recover hidden\\[0.5mm]entry $d_j=\det(K_j)$};

\node[font=\Large\bfseries, above=0.3cm of rec] {Decoding Process};

\draw[arr, dashed, rounded corners=8pt] (KM.west) -- ++(-2.5, 0) |- (rec.west);

\path (KM.south) -- (rec.north) node[midway, left=2.6cm, font=\small\bfseries, align=right] {Channel\\[0.5mm]Transmission};

\node (Kj) [box, below=of rec] {$K_j$};
\node (inv) [op, below=of Kj] {$(Q^{a_{j-1}})^{-1}\pmod{\mathbf{P}}$};
\node (dec) [op, below=of inv] {$\widehat{M}_j \equiv K_j(Q^{a_{j-1}})^{-1}\pmod{\mathbf{P}}$};
\node (Mhh) [box, below=of dec] {$\widehat{M}$};
\node (invmap) [op, below=of Mhh] {$M=\left(\rho_P|_{\mathcal{C}}\right)^{-1}(\widehat{M})$};
\node (MM) [box, below=of invmap] {$M$};
\node (pt2) [box, below=of MM] {Plaintext};

\draw[arr] (rec) -- (Kj);
\draw[arr] (Kj) -- (inv);
\draw[arr] (inv) -- (dec);
\draw[arr] (dec) -- (Mhh);
\draw[arr] (Mhh) -- (invmap);
\draw[arr] (invmap) -- (MM);
\draw[arr] (MM) -- (pt2);

\draw[arr] (keys.south) |- (inv.east);

\end{tikzpicture}
}
\caption{Schematic diagram of the FLLJ-POLY encryption and decryption procedure.}
\label{fig:flljpoly-schematic-vertical}
\end{figure}

\subsection{An Illustrative Example}
\label{subsec:2b}

\begin{example}
\label{ex:main}
We illustrate all steps of the FLLJ-POLY Alphabetic Encryption Algorithm with $\mathbf{P=10159}$.

{Coding process:} \medskip

{Step 1.} Message text:
$\text{A LEONARDO NUMBER IS AN ODD NUMBER.}$

{Step 2.} Message matrix $M$:
\begin{equation*}
M=\left(
\begin{array}{cccccc}
A & * & L & E & O & N \\
A & R & D & O & * & N \\
U & M & B & E & R & * \\
I & S & * & A & N & * \\
O & D & D & * & N & U \\
M & B & E & R & . & *%
\end{array}%
\right) _{6\times 6}.
\end{equation*}

{Step 3.} The number of $3\times 3$ submatrices is $k=4$. With pre-shared parameters $s=7$ and $p=47$, since $\mathcal{L}_{7}=41$, the initial key is
\begin{equation*}
a_{0} = \left(41\cdot 4\right)\bmod{47}+1 = 23+1 = 24.
\end{equation*}
The subsequent block keys are generated via $a_{i+1} = (41\, a_i + a_i^2 + i)(\bmod 47)$:
\begin{align*}
a_1 &= (41\cdot 24 + 24^2 + 0)(\bmod 47) = (984+576)(\bmod 47) = 1560(\bmod 47)= 9,\\
a_2 &= (41\cdot 9 + 9^2 + 1)(\bmod 47) = (369+81+1)(\bmod 47) = 451(\bmod 47) = 28,\\
a_3 &= (41\cdot 28 + 28^2 + 2)(\bmod 47) = (1148+784+2)(\bmod 47) = 1934(\bmod 47) = 7.
\end{align*}
Thus the four block keys are $(a_0, a_1, a_2, a_3) = (24, 9, 28, 7)$, and the Lucas separator index uses $a_0 = 24$.

{Step 4.} The three character tables (unreduced values shown for reference) are given in Tables~\ref{tab:1}--\ref{tab:2b}.

\begin{table}[H]
\centering
\caption{\small First character table: Fibonacci numbers (from $F_{5}$)}
\label{tab:1}
\resizebox{\textwidth}{!}{
\begin{tabular}{|c|c|c|c|c|c|c|c|c|c|}
\hline
A & B & C & D & E & F & G & H & I & J \\ \hline
5 & 8 & 13 & 21 & 34 & 55 & 89 & 144 & 233 & 377 \\ \hline
K & L & M & N & O & P & Q & R & S & T \\ \hline
610 & 987 & 1597 & 2584 & 4181 & 6765 & 10946 & 17711 & 28657 & 46368 \\ \hline
U & V & W & X & Y & Z & 1 & 2 & . & , \\ \hline
75025 & 121393 & 196418 & 317811 & 514229 & 832040 & 1346269 & 2178309 & 3524578 & 5702887 \\ \hline
\end{tabular}
}
\end{table}

\begin{table}[H]
\centering
\caption{\small Second character table: Leonardo numbers (from $\mathcal{L}_{5}$)}
\label{tab:2}
\resizebox{\textwidth}{!}{
\begin{tabular}{|c|c|c|c|c|c|c|c|c|c|}
\hline
A & B & C & D & E & F & G & H & I & J \\ \hline
15 & 25 & 41 & 67 & 109 & 177 & 287 & 465 & 753 & 1219 \\ \hline
K & L & M & N & O & P & Q & R & S & T \\ \hline
1973 & 3193 & 5167 & 8361 & 13529 & 21891 & 35421 & 57313 & 92735 & 150049 \\ \hline
U & V & W & X & Y & Z & 1 & 2 & . & , \\ \hline
242785 & 392835 & 635621 & 1028457 & 1664079 & 2692537 & 4356617 & 7049155 & 11405773 & 18454929 \\ \hline
\end{tabular}
}
\end{table}

\begin{table}[H]
\centering
\caption{\small Third character table: Jacobsthal numbers (from $J_{7}$)}
\label{tab:2b}
\resizebox{\textwidth}{!}{
\begin{tabular}{|c|c|c|c|c|c|c|c|c|c|}
\hline
A & B & C & D & E & F & G & H & I & J \\ \hline
43 & 85 & 171 & 341 & 683 & 1365 & 2731 & 5461 & 10923 & 21845 \\ \hline
K & L & M & N & O & P & Q & R & S & T \\ \hline
43691 & 87381 & 174763 & 349525 & 699051 & 1398101 & 2796203 & 5592405 & 11184811 & 22369621 \\ \hline
U & V & W & X & Y & Z & 1 & 2 & . & , \\ \hline
44739243 & 89478485 & 178956971 & 357913941 & 715827883 & 1431655765 & 2863311531 & 5726623061 & 11453246123 & 22906492245 \\ \hline
\end{tabular}
}
\end{table}
\noindent{Step 5.}
We apply the selector
\[
\sigma=\bigl(v^{2}+(i-1)+(j-1)\bigr)(\bmod 3),
\qquad 1\leq i,j\leq 3m.
\]
For example:
\begin{itemize}
\item Position $(1,1)$: ``A'' with $n=1$,
\[
\sigma=(1+0+0)(\bmod 3)=1
\Rightarrow \text{Fibonacci: } 5.
\]

\item Position $(2,1)$: ``A'' with $n=2$,
\[
\sigma=(4+1+0)(\bmod 3)=2
\Rightarrow \text{Jacobsthal: } 43.
\]

\item Position $(4,4)$: ``A'' with $n=3$,
\[
\sigma=(9+3+3)(\bmod 3)=0
\Rightarrow \text{Leonardo: } 15.
\]
\end{itemize}

It is worth noting that the selector $\sigma$ does not assign alphabets
in a fixed cyclic order. A linear selector depending only on the occurrence
count would cycle through the three alphabets periodically, producing a
predictable substitution pattern. The combined dependence on $v^2$, $i$,
and $j$ breaks this regularity: as seen above, the three occurrences of
``A'' are mapped to the Fibonacci, Jacobsthal, and Leonardo alphabets in
that order, rather than the fixed cycle $0\to1\to2$.

For *'s, since $a_0=24$, we use $L_{z_i}\bmod P$ with $z_i = 24+t_i$. The auxiliary table (with modular values) is:

\begin{table}[h]
\caption{\small Auxiliary table of Lucas numbers for word separators}
\label{tab:3}
{\footnotesize {
\begin{tabular}{|c|c|c|}
\hline
Order of appearance of $*$ & $L_{z_i}$ (unreduced) & $L_{z_i}\bmod 10159$ \\ \hline
1 & $L_{25}=167761$ & $5217$ \\ \hline
2 & $L_{26}=271443$ & $7309$ \\ \hline
3 & $L_{27}=439204$ & $2367$ \\ \hline
4 & $L_{28}=710647$ & $9676$ \\ \hline
5 & $L_{29}=1149851$ & $1884$ \\ \hline
6 & $L_{30}=1860498$ & $1401$ \\ \hline
7 & $L_{31}=3010349$ & $3285$ \\ \hline
\end{tabular}
}}
\end{table}

After applying the selector rule and reducing all values modulo $\mathbf{P=10159}$, we obtain:
\begin{equation*}
\widehat{M} = {M} (\bmod 10159) = \left(
\begin{array}{cccccc}
5 & 5217 & 3193 & 34 & 8239 & 8361 \\
43 & 6518 & 21 & 8239 & 7309 & 2584 \\
9128 & 1597 & 85 & 109 & 7552 & 2367 \\
233 & 9911 & 9676 & 15 & 2584 & 1884 \\
4181 & 67 & 67 & 1401 & 8361 & 3912 \\
5167 & 8 & 34 & 4955 & 9564 & 3285
\end{array}
\right).
\end{equation*}
{Step 6.} We divide $\widehat{M}$ into four $3\times 3$ submatrices
\[
\widehat{M}_1,\ \widehat{M}_2,\ \widehat{M}_3,\ \widehat{M}_4,
\]
namely
\[
\widehat{M}_1=
\left(
\begin{array}{ccc}
5 & 5217 & 3193 \\
43 & 6518 & 21 \\
9128 & 1597 & 85
\end{array}
\right),
\qquad
\widehat{M}_2=
\left(
\begin{array}{ccc}
34 & 8239 & 8361 \\
8239 & 7309 & 2584 \\
109 & 7552 & 2367
\end{array}
\right),
\]
\[
\widehat{M}_3=
\left(
\begin{array}{ccc}
233 & 9911 & 9676 \\
4181 & 67 & 67 \\
5167 & 8 & 34
\end{array}
\right),
\qquad
\widehat{M}_4=
\left(
\begin{array}{ccc}
15 & 2584 & 1884 \\
1401 & 8361 & 3912 \\
4955 & 9564 & 3285
\end{array}
\right).
\]

{Step 7.} Using the block keys
\[
(a_0,a_1,a_2,a_3)=(24,9,28,7),
\]
the ciphertext blocks are computed by
\[
K_1 \equiv \widehat{M}_1Q^{24}\pmod{10159},\qquad
K_2 \equiv \widehat{M}_2Q^{9}\pmod{10159},
\]
\[
K_3 \equiv \widehat{M}_3Q^{28}\pmod{10159},\qquad
K_4 \equiv \widehat{M}_4Q^{7}\pmod{10159}.
\]
Thus,
\[
\begin{aligned}
K_1 &\equiv
\begin{pmatrix}
4236 & 4477 & 8164 \\
1512 & 7415 & 593 \\
8721 & 7392 & 4042
\end{pmatrix},
&
K_2 &\equiv
\begin{pmatrix}
2632 & 900 & 8139 \\
7513 & 9052 & 115 \\
9023 & 8441 & 233
\end{pmatrix}
\pmod{10159},\\[2mm]
K_3 &\equiv
\begin{pmatrix}
8733 & 9140 & 8788 \\
2045 & 4576 & 3581 \\
8905 & 6255 & 7684
\end{pmatrix},
&
K_4 &\equiv
\begin{pmatrix}
8866 & 2920 & 240 \\
2836 & 9887 & 3821 \\
8521 & 3365 & 2891
\end{pmatrix}
\pmod{10159}.
\end{aligned}
\]
\medskip{Step 8.} The determinants are
\[
\begin{aligned}
\det(K_1)&\equiv 7735 \pmod{10159},&
\det(K_2)&\equiv 4898 \pmod{10159},\\
\det(K_3)&\equiv 557 \pmod{10159},&
\det(K_4)&\equiv 10058 \pmod{10159}.
\end{aligned}
\]
Since the entries of $K_j$ are determined by the matrix product
$\widehat{M}_j Q^{a_{j-1}} \pmod{\mathbf{P}}$, the cofactor $A$ of
the hidden entry vanishes modulo $\mathbf{P}$ only if the remaining
entries satisfy the special congruence $K_{11}K_{33} \equiv K_{13}K_{31}
\pmod{\mathbf{P}}$. This is a degenerate condition that does not arise
in the worked examples and is not expected to occur in typical
plaintexts. In the event that this congruence holds, the entry
corresponding to $(k+1)\bmod 9$ in the row-wise ordering is selected
as the hidden entry instead. Since the receiver applies the same
fallback rule, the decoding procedure remains unambiguous.

\medskip
{Step 9.} Since $k=4$, we have $k\equiv 4 \pmod 9$. Hence the $4$th entry in the row-wise ordering of each block is hidden; equivalently, the hidden element is the $(2,1)$-entry of each $3\times 3$ block.

{Step 10.} Therefore, the code matrix $K_M$ is obtained by placing in each row first the determinant of the corresponding block, followed by the remaining eight visible entries written row by row. Thus,
\[
K_M=
\left(
\begin{array}{ccccccccc}
7735  & 4236 & 4477 & 8164 & 7415 & 593  & 8721 & 7392 & 4042 \\
4898  & 2632 & 900  & 8139 & 9052 & 115  & 9023 & 8441 & 233 \\
557   & 8733 & 9140 & 8788 & 4576 & 3581 & 8905 & 6255 & 7684 \\
10058 & 8866 & 2920 & 240  & 9887 & 3821 & 8521 & 3365 & 2891
\end{array}
\right)\pmod{10159}.
\]
This matrix is sent to the receiver.

\medskip
\textbf{Decoding process:} \medskip

\noindent{Step 1.}
To initiate decryption, the receiver must reconstruct the private key sequence using the pre-shared secret seed $s=7$ (yielding the Leonardo number $\mathcal{L}_7 = 41$) and the prime $p=47$. Given the transmitted block parameter $k=4$, the receiver computes the sequence explicitly:
\begin{align}
a_0 &= \left(41 \cdot 4\right)(\bmod 47) + 1 = 24, \label{eq:rec_a0_dec} \\
a_{i+1} &= \left(41 \cdot a_i + a_i^2 + i\right)(\bmod 47), \quad \text{for } i=0,1,2. \label{eq:rec_ai_dec}
\end{align}
This operation successfully yields the identical sequence $(a_0,a_1,a_2,a_3)=(24,9,28,7)$.

Once the sequence is recovered, the receiver computes the corresponding modular inverse matrices$(Q^{a_{j-1}})^{-1} \pmod{\mathbf{P}}$ for each block, ($j=1, \dots, 4$). With $\mathbf{P=10159}$, the inverse matrices are explicitly computed as:
\[
(Q^{24})^{-1}(\bmod 10159)=
\left(
\begin{array}{ccc}
786 & 2606 & 8338 \\
1821 & 5733 & 6248 \\
7553 & 1821 & 5733
\end{array}
\right),
\]
\[
(Q^{9})^{-1}(\bmod 10159)=
\left(
\begin{array}{ccc}
10150 & 12 & 10137 \\
22 & 10126 & 56 \\
10147 & 22 & 10126
\end{array}
\right),
\]
\[
(Q^{28})^{-1}(\bmod 10159)=
\left(
\begin{array}{ccc}
3911 & 514 & 3396 \\
6763 & 2883 & 3881 \\
9645 & 6763 & 2883
\end{array}
\right),
\]
\[
(Q^{7})^{-1}(\bmod 10159)=
\left(
\begin{array}{ccc}
10155 & 4 & 10150 \\
9 & 10147 & 22 \\
10155 & 9 & 10147
\end{array}
\right).
\]
{Step 2.} Since the hidden entry in each block is the $(2,1)$-entry, the missing values are recovered from the determinant conditions
\[
\begin{aligned}
\det(K_1)&\equiv 7735 \pmod{10159},&
\det(K_2)&\equiv 4898 \pmod{10159},\\
\det(K_3)&\equiv 557 \pmod{10159},&
\det(K_4)&\equiv 10058 \pmod{10159}.
\end{aligned}
\]
Thus, the hidden entries are determined as
\[
(K_1)_{21}=1512,\qquad
(K_2)_{21}=7513,\qquad
(K_3)_{21}=2045,\qquad
(K_4)_{21}=2836.
\]
For example, to recover $(K_1)_{21}$, we solve
\[
\det\!\left(
\begin{array}{ccc}
4236 & 4477 & 8164 \\
x & 7415 & 593 \\
8721 & 7392 & 4042
\end{array}
\right)
\equiv 7735 \pmod{10159}.
\]
A direct expansion gives
\[
973x+9614 \equiv 7735 \pmod{10159},
\]
and hence
\[
x \equiv 1512 \pmod{10159}.
\]
Therefore $(K_1)_{21}=1512$. Proceeding similarly for the remaining blocks yields the hidden entries listed above.

Hence the four ciphertext blocks are reconstructed as
\[
\begin{aligned}
K_1 &\equiv
\begin{pmatrix}
4236 & 4477 & 8164 \\
1512 & 7415 & 593 \\
8721 & 7392 & 4042
\end{pmatrix},
&
K_2 &\equiv
\begin{pmatrix}
2632 & 900 & 8139 \\
7513 & 9052 & 115 \\
9023 & 8441 & 233
\end{pmatrix},
\\[1.2em]
K_3 &\equiv
\begin{pmatrix}
8733 & 9140 & 8788 \\
2045 & 4576 & 3581 \\
8905 & 6255 & 7684
\end{pmatrix},
&
K_4 &\equiv
\begin{pmatrix}
8866 & 2920 & 240 \\
2836 & 9887 & 3821 \\
8521 & 3365 & 2891
\end{pmatrix}
\pmod{10159}.
\end{aligned}
\]

{Step 3.} The plaintext blocks are recovered by
\[
\widehat{M}_1 \equiv K_1 (Q^{24})^{-1}\pmod{10159},
\qquad
\widehat{M}_2 \equiv K_2 (Q^{9})^{-1}\pmod{10159},
\]
\[
\widehat{M}_3 \equiv K_3 (Q^{28})^{-1}\pmod{10159},
\qquad
\widehat{M}_4 \equiv K_4 (Q^{7})^{-1}\pmod{10159}.
\]
This yields
\[
\begin{aligned}
\widehat{M}_1 &\equiv
\begin{pmatrix}
5 & 5217 & 3193 \\
43 & 6518 & 21 \\
9128 & 1597 & 85
\end{pmatrix},
&
\widehat{M}_2 &\equiv
\begin{pmatrix}
34 & 8239 & 8361 \\
8239 & 7309 & 2584 \\
109 & 7552 & 2367
\end{pmatrix},
\\[1.2em]
\widehat{M}_3 &\equiv
\begin{pmatrix}
233 & 9911 & 9676 \\
4181 & 67 & 67 \\
5167 & 8 & 34
\end{pmatrix},
&
\widehat{M}_4 &\equiv
\begin{pmatrix}
15 & 2584 & 1884 \\
1401 & 8361 & 3912 \\
4955 & 9564 & 3285
\end{pmatrix}
.
\end{aligned}
\]

{Step 4.} Reassembling these blocks gives
\[
\widehat{M} \equiv
\left(
\begin{array}{cccccc}
5 & 5217 & 3193 & 34 & 8239 & 8361 \\
43 & 6518 & 21 & 8239 & 7309 & 2584 \\
9128 & 1597 & 85 & 109 & 7552 & 2367 \\
233 & 9911 & 9676 & 15 & 2584 & 1884 \\
4181 & 67 & 67 & 1401 & 8361 & 3912 \\
5167 & 8 & 34 & 4955 & 9564 & 3285
\end{array}
\right).
\]

{Step 5.} We construct the message matrix $M$ by converting the entries of $\widehat{M}$ through the three character tables and the Lucas separator values:
\[
M=
\left(
\begin{array}{cccccc}
A & * & L & E & O & N \\
A & R & D & O & * & N \\
U & M & B & E & R & * \\
I & S & * & A & N & * \\
O & D & D & * & N & U \\
M & B & E & R & . & *
\end{array}
\right).
\]

{Step 6.} Therefore the recovered plaintext is
\[
\text{A LEONARDO NUMBER IS AN ODD NUMBER.}
\]
\end{example}

\section{Statistical cryptanalysis: Friedman test, chi-square test, and Shannon entropy}
\label{sec:3}

This section examines the statistical behaviour of the proposed FLLJ-POLY
scheme against classical ciphertext-only frequency analysis. We use three
standard indicators: the Friedman test \cite{Friedman1922}, or Index of
Coincidence (I.C.), the chi-square \((\chi^2)\) goodness-of-fit test~\cite{Cochran1952}, and Shannon
entropy\cite{Shannon2021}.

For standard English text, the I.C.\ is \(\kappa_p\approx 0.0667\)
\cite{Stinson2005, Menezes1996}, while for uniformly random text over a
26-letter alphabet it is
\[
\kappa_r=\frac{1}{26}\approx 0.0385
\]
\begin{remark}\label{rem:alphabet-size}
The value \(\kappa_r=1/26\) applies to a 26-symbol alphabet. Since the proposed
scheme expands the effective ciphertext alphabet, the corresponding random
reference value should be interpreted as \(1/c\), where \(c\) denotes the
number of distinct ciphertext symbols observed in the encrypted text.
\end{remark}

\subsection{Index of Coincidence}
\label{subsec:4a}

A standard English plaintext of approximately 500 characters was processed. When placed into a $24\times 24$ message matrix (with spaces replaced by $*$ and padding $*$'s appended), the total symbol count is $N=576$. The plaintext is:

\begin{quote}
\textit{``THIS STUDY INTRODUCES A NOVEL HOMOPHONIC SUBSTITUTION CIPHER DESIGNED TO RESIST FREQUENCY ANALYSIS ATTACKS. THE ALGORITHM UTILIZES THREE DISTINCT SEQUENCES, LEONARDO, FIBONACCI, AND JACOBSTHAL, TO ENCRYPT CHARACTERS BASED ON POSITION AND OCCURRENCE COUNT. MOREOVER, SPACE CHARACTERS ARE MASKED USING A NONPERIODIC LUCAS SEQUENCE. THIS STRUCTURE AIMS TO FLATTEN THE FREQUENCY DISTRIBUTION, THEREBY LOWERING THE INDEX OF COINCIDENCE. THE PROPOSED METHOD DEMONSTRATES SIGNIFICANT IMPROVEMENTS IN STATISTICAL SECURITY COMPARED TO STANDARD MONOALPHABETIC CIPHERS.''}
\end{quote}

The Friedman test calculates the probability that two randomly selected symbols from the ciphertext are identical:
\begin{equation}
    \text{I.C.} = \frac{\sum_{i=1}^{c} n_i(n_i - 1)}{N(N - 1)}, \label{eq:IC}
\end{equation}
where $n_i$ is the frequency of the $i$-th symbol, $N$ is the total length, and $c$ is the effective alphabet size~\cite{Friedman1922, Stinson2005}.

For the mod~3 scheme, each letter is distributed across three
sequences (Leonardo, Fibonacci, Jacobsthal) via the selector
$\sigma=(v^{2}+(i-1)+(j-1)) (\bmod 3)$. For example, the letter ``E'' (54 occurrences) is distributed
as $(18,18,18)$ across the three alphabets, yielding a total
Friedman contribution of
\[
18\cdot 17 + 18\cdot 17 + 18\cdot 17 = 918.
\]
By contrast, a two-alphabet equal split $(27,27)$ would yield
$27\cdot 26 + 27\cdot 26 = 1404$, so using three alphabets
instead of two reduces this contribution by approximately
$35\%$ for this letter.

\begin{table}[t]%
\centering
\caption{Top-frequency ciphertext symbols for the position-dependent scheme.\label{tab:freq_mod3_short}}%
\begin{tabular*}{425pt}{@{\extracolsep\fill}llcc@{\extracolsep\fill}}
\toprule
\textbf{Ciphertext Symbol} & \textbf{Source} & \textbf{Freq. ($n_i$)} & \textbf{$n_i(n_i-1)$} \\
\midrule
109  & E (Leo) & 22 & 462 \\
683  & E (Jac) & 18 & 306 \\
34   & E (Fib) & 14 & 182 \\
7823 & T (Leo) & 19 & 342 \\
9662 & T (Jac) & 16 & 240 \\
5732 & T (Fib) & 12 & 132 \\
753  & I (Leo) & 16 & 240 \\
764  & I (Jac) & 13 & 156 \\
233  & I (Fib) & 11 & 110 \\
1304 & S (Leo) & 15 & 210 \\
9911 & S (Jac) & 13 & 156 \\
8339 & S (Fib) & 11 & 110 \\
\midrule
\textbf{Others} & \textbf{Lower frequencies ($n_i \le 10$)} & \textbf{--} & \textbf{--} \\
\midrule
\textbf{Word Separators} & \textbf{Lucas numbers ($n_i=1$ each)} & \textbf{88} & \textbf{0} \\
\midrule
\textbf{Total} & & \textbf{576} & \textbf{5054} \\
\bottomrule
\end{tabular*}
\end{table}

Table \ref{tab:freq_mod3_short} lists the most frequent ciphertext symbols contributing to the Friedman statistic. The remaining symbols occur with relatively low frequencies (typically $n_i \leq 10$) and therefore have a limited impact on the overall Index of Coincidence.

In total, 158 distinct ciphertext symbols are observed in the position-dependent mod~3 scheme, including 88 unique Lucas separators, each appearing exactly once and thus contributing zero to the Friedman sum.

This distribution shows that the position-dependent mod~3 rule spreads high-frequency plaintext
symbols across several numerical alphabets instead of concentrating them in a
single ciphertext symbol.

\begin{equation}
    \text{I.C.}_{\bmod 3} = \frac{5054}{576 \times 575} = \frac{5054}{331200} \approx \mathbf{0.01526}, \label{eq:IC-mod3}
\end{equation}

A comparison of these I.C. values with standard reference models is provided in Table \ref{tab:comparison}.

\begin{table}[h]
\centering
\caption{Comparative analysis of I.C. values with uniform reference baselines}
\label{tab:comparison}
\begin{tabular}{l c}
\toprule
\textbf{Text Type / Reference Model} & \textbf{Typical I.C.} \\
\midrule
Standard English & $0.0667$ \\
Monoalphabetic Cipher & $0.0667$ \\
Vigen\`{e}re Cipher (Short Key) & $\approx 0.0450$ \\
Uniform random text over $26$ symbols & $1/26 \approx 0.0385$ \\
Uniform random text over $158$ symbols & $1/158 \approx 0.0063$ \\
\textbf{Proposed position-dependent mod $3$ scheme} & $\mathbf{0.0153}$ \\
\bottomrule
\end{tabular}
\end{table}

\subsection{Chi-square goodness-of-fit test}
\label{subsec:4b}

The chi-square \((\chi^2)\) goodness-of-fit test is used to measure how closely
the observed ciphertext symbol frequencies follow a uniform distribution~\cite{Cochran1952}.
In this context, a lower \(\chi^2\) value indicates that the frequency profile
of the ciphertext is closer to uniform, making simple frequency-based
distinguishability less pronounced. We use Pearson's chi-square statistic:
\begin{equation}
\chi^2 = \sum_{i=1}^{c} \frac{(O_i - E_i)^2}{E_i}.
\label{eq:chi2}
\end{equation}

With \(c=158\) distinct ciphertext symbols and \(N=576\), the expected
frequency under the uniform model is
\[
E_i=\frac{N}{c}=\frac{576}{158}\approx 3.65.
\]
Using the observed ciphertext frequencies \(O_i\), Pearson's statistic gives
\[
\chi^2=\sum_{i=1}^{c}\frac{(O_i-E_i)^2}{E_i}\approx 968.54.
\]
Normalizing by the number of distinct symbols, we obtain
\[
\chi^2/c\approx 6.13.
\]
This normalized statistic serves as a scale-adjusted measure of deviation from the uniform frequency profile. For comparison, a standard monoalphabetic substitution applied to the same
plaintext yields \(\chi^2/c\approx 20.83\). Thus, the proposed position-dependent mod~3 scheme
reduces the normalized chi-square value by approximately
\[
\frac{20.83-6.13}{20.83}\times 100 \approx 70.6\%.
\]
Equivalently, the monoalphabetic value is about \(3.4\) times larger than the
corresponding mod~3 value.
\subsection{Shannon entropy analysis}
\label{subsec:4c}

Shannon entropy measures the uncertainty of a symbol distribution. In the
present context, it is used to quantify how evenly the ciphertext symbols are
distributed: higher entropy indicates a less concentrated frequency profile,
whereas lower entropy indicates dominance by a smaller number of symbols.
Shannon entropy~\cite{Shannon2021} is defined by
\begin{equation}
H=-\sum_{i=1}^{c} p_i\log_2 p_i,
\label{eq:entropy}
\end{equation}
where \(p_i=n_i/N\) is the observed probability of the \(i\)-th ciphertext
symbol.

The entropy satisfies
\[
0\leq H\leq H_{\max},
\qquad
H_{\max}=\log_2 c.
\]
The value \(H=0\) occurs when only one ciphertext symbol appears, while
\(H=H_{\max}\) occurs when all \(c\) symbols are equally likely. Thus, values
closer to \(H_{\max}\) indicate a more balanced ciphertext frequency
distribution.

For the proposed position-dependent mod~3 scheme, \(c=158\) distinct ciphertext symbols are
observed. Hence, the maximum possible entropy is
\[
H_{\max}=\log_2 158\approx 7.304\ \text{bits}.
\] The computed entropy is
\[
H_{\bmod 3}\approx 6.416 \ \text{bits}.
\]
Therefore, the entropy efficiency is
\[
\eta_{\bmod 3}
=
\frac{H_{\bmod 3}}{H_{\max}}\times 100
=
\frac{6.416}{7.304}\times 100
\approx 87.8\%.
\]

This indicates that the ciphertext symbols are relatively well distributed
over the effective alphabet, with reduced concentration around dominant
symbols.
\subsection{Computational performance benchmarks}
\label{subsec:4d}

To evaluate the practical efficiency of the proposed FLLJ-POLY scheme,
computational benchmarks were conducted using a C/C++ implementation compiled
with GCC under the \texttt{-O3} optimization flag. The reported values are
intended as approximate performance estimates, since execution times may vary
depending on the hardware configuration, compiler version, and system-level
conditions. For the unreduced large-integer arithmetic, the GNU Multiple
Precision Arithmetic Library (GMP) was used, whereas the proposed modular
framework with \(\mathbf{P}=10159\) operated entirely within native 64-bit
unsigned integer registers.

Table~\ref{tab:combined_perf} presents the encryption and decryption times for
three different message lengths, comparing the unreduced implementation with
the modular scheme.

\begin{table}[H]
\centering
\caption{Performance comparison in C/C++: unreduced GMP arithmetic versus the modular scheme with \(\mathbf{P}=10159\)}
\label{tab:combined_perf}
\resizebox{\textwidth}{!}{
\begin{tabular}{l c c c c c}
\toprule
\textbf{Message} & \textbf{Matrix} & \textbf{Unreduced} & \textbf{Modular} & \textbf{Modular} & \textbf{Encode} \\
\textbf{Length} & \textbf{Size} & \textbf{Encode} & \textbf{Encode} & \textbf{Decode} & \textbf{Speedup} \\
\midrule
Short \((\sim 36\) chars)  & \(6\times 6\)   & \(0.093\) ms & \(0.012\) ms & \(0.014\) ms & \(\sim 8\times\)  \\
Medium \((\sim 150\) chars) & \(12\times 12\) & \(0.594\) ms & \(0.076\) ms & \(0.084\) ms & \(\sim 7.8\times\) \\
Long \((\sim 500\) chars)  & \(24\times 24\) & \(1.473\) ms & \(0.213\) ms & \(0.234\) ms & \(\sim 7\times\) \\
\bottomrule
\end{tabular}
}
\end{table}

The results in Table~\ref{tab:combined_perf} indicate that the modular
reduction significantly improves computational efficiency. In the unreduced
setting, the matrix entries require large-integer arithmetic, whereas the
modular scheme restricts all computations to \(\mathbb{Z}_{\mathbf{P}}\).
Across the tested message sizes, the modular encoding process achieves an
approximately \(7\)--\(8\times\) speedup. For the \(24\times 24\) message matrix,
the encoding time decreases from \(1.473\) ms to \(0.213\) ms, corresponding to
an approximate \(7\times\) improvement.

\subsection{Avalanche effect analysis}
\label{subsec:avalanche}

The avalanche effect is an empirical measure of the sensitivity of an
encryption algorithm to small changes in the input or in the secret key
\cite{Upadhyay2022, Shannon1949}. Ideally, a one-bit change should affect approximately half
of the ciphertext bits, as expressed by the Strict Avalanche Criterion
\cite{Upadhyay2022,Menezes1996,WebsterTavares1985,Stallings2017,Feistel1973}.

We measure the avalanche ratio by the normalized Hamming distance
\[
\mathrm{AE}
=
\frac{d_H(C,C')}{N}\times 100\%,
\]
where \(d_H(C,C')\) denotes the number of differing bits between the binary
representations of the original ciphertext \(C\) and the perturbed ciphertext
\(C'\), and \(N\) is the total number of ciphertext bits. Since all entries of
the code matrix \(K_M\) are reduced modulo \(\mathbf{P}=10159\), each entry is represented
using
\[
\lceil \log_2 P\rceil=14
\]
bits. In the worked example, \(K_M\) has \(4\times 9=36\) entries, and hence
\[
N=36\cdot 14=504.
\]

\subsubsection{Computation of the Hamming distance}

Given two code matrices \(K_M\) and \(K_M'\) corresponding to two
different key configurations (or two different plaintexts), the
Hamming distance \(d_H(C,C')\) is computed as follows.

Each integer entry \(x \in \{0,1,\ldots,P-1\}\) of the code matrix
is written as a fixed-length 14-bit binary string
\[
x = \sum_{l=0}^{13} b_l \, 2^l, \qquad b_l \in \{0,1\},
\]
padded with leading zeros if necessary so that every entry occupies
exactly 14 bits. For example,
\[
7735 = (01111000110111)_2, \qquad
4236 = (01000010001100)_2.
\]
Concatenating all 36 entries row-wise yields two binary strings
\(C\) and \(C'\), each of length \(N = 504\) bits.

The bitwise exclusive-or \(C \oplus C'\) is computed. Each bit
position where \(C\) and \(C'\) differ produces a~\(1\) in the
result, and each position where they agree produces a~\(0\). Hence
\[
d_H(C,C') = \mathrm{wt}(C \oplus C'),
\]
where \(\mathrm{wt}(\cdot)\) denotes the Hamming weight, i.e., the
number of \(1\)-bits in the binary string. At the level of individual
entries, if \(x_i\) and \(x_i'\) denote the \(i\)-th entries of
\(K_M\) and \(K_M'\) respectively, then
\[
d_H(C,C') = \sum_{i=1}^{36} \mathrm{wt}(x_i \oplus x_i'),
\]
where each XOR is taken over the 14-bit representations.

For an illustrative example, consider the first entry of \(K_M\) under the base configuration
\((s,p)=(7,47)\) and under the perturbed configuration \((s,p)=(8,47)\):
\[
x_1 = 7735 = (01111000110111)_2,
\qquad
x_1' = 3189 = (00110001110101)_2.
\]
Their bitwise XOR is
\[
x_1 \oplus x_1' = (0 1 0 0 1 0 0 1 0 0 0 0 1 0)_2,
\]
which has Hamming weight \(\mathrm{wt} = 4\). This entry alone
contributes \(4\) to the total \(d_H\). Summing this contribution
over all 36 entries yields the overall Hamming distance \(d_H(C,C')\)
reported in Tables~\ref{tab:ae-seed} and~\ref{tab:ae-prime}.

\subsubsection{Key avalanche effect}

We first fix the plaintext and perturb the pre-shared secret key parameters.
Two cases are considered: changing the seed \(s\) while keeping \(p\) fixed,
and changing the prime \(p\) while keeping \(s\) fixed. The base configuration
is
\[
(s,p)=(7,47),
\]
which yields the block-key sequence
\[
(a_0,a_1,a_2,a_3)=(24,9,28,7).
\]

For example, when the seed is changed from \(s=7\) to \(s=8\), the resulting
ciphertext differs from the original one in \(235\) out of \(504\) bits. Thus,
\[
\mathrm{AE}_K
=
\frac{235}{504}\times 100\%
\approx 46.63\%.
\]
The remaining key perturbations are evaluated in the same way.

\begin{table}[H]
\centering
\caption{Key avalanche effect under seed perturbation with \(p=47\) fixed}
\label{tab:ae-seed}
\begin{tabular}{l c c c}
\toprule
\textbf{Perturbation} & \textbf{New key sequence} &
\(\boldsymbol{d_H}\) & \(\boldsymbol{\mathrm{AE}_K}\) \\
\midrule
\(s=7\to s=8\)  & \((34,\,3,\,23,\,4)\)   & \(235\) & \(46.63\%\) \\
\(s=7\to s=9\)  & \((14,\,30,\,35,\,13)\)  & \(230\) & \(45.63\%\) \\
\(s=7\to s=10\) & \((4,\,19,\,12,\,14)\)   & \(258\) & \(51.19\%\) \\
\(s=7\to s=12\) & \((28,\,33,\,32,\,20)\)  & \(234\) & \(46.43\%\) \\
\(s=7\to s=17\) & \((36,\,13,\,37,\,38)\)  & \(258\) & \(51.19\%\) \\
\midrule
\multicolumn{3}{l}{\textbf{Average}} & \(\mathbf{48.21\%}\) \\
\bottomrule
\end{tabular}
\end{table}

\begin{table}[H]
\centering
\caption{Key avalanche effect under prime perturbation with \(s=7\) fixed}
\label{tab:ae-prime}
\begin{tabular}{l c c c}
\toprule
\textbf{Perturbation} & \textbf{New key sequence} &
\(\boldsymbol{d_H}\) & \(\boldsymbol{\mathrm{AE}_K}\) \\
\midrule
\(p=47\to p=53\) & \((6,\,17,\,33,\,6)\)    & \(237\) & \(47.02\%\) \\
\(p=47\to p=59\) & \((47,\,6,\,47,\,8)\)    & \(277\) & \(54.96\%\) \\
\(p=47\to p=61\) & \((43,\,13,\,32,\,20)\)  & \(248\) & \(49.21\%\) \\
\(p=47\to p=67\) & \((31,\,21,\,30,\,55)\)  & \(214\) & \(42.46\%\) \\
\(p=47\to p=71\) & \((23,\,52,\,9,\,26)\)   & \(255\) & \(50.60\%\) \\
\midrule
\multicolumn{3}{l}{\textbf{Average}} & \(\mathbf{48.85\%}\) \\
\bottomrule
\end{tabular}
\end{table}

The average key avalanche ratios are \(48.21\%\) for seed perturbations and
\(48.85\%\) for prime perturbations. Hence, in this example, small changes in
either component of the secret pair \((s,p)\) produce ciphertext changes close
to the ideal \(50\%\) level.

\subsection{Key-space considerations and security}
\label{subsec:keyspace}

This section discusses the effective key space of the proposed scheme and its implications for exhaustive search. In our scheme, the secret key material is the pair
\[
K=(s,p),
\]
where $s$ is a seed and $p$ is a prime modulus, both pre-shared
between the sender and the receiver. The public block number $k$
is not part of the secret key, as it is determined by the message
matrix size.

Since the key generation recurrence uses $\mathcal{L}_s$ only
modulo $p$, the effective contribution of the seed is captured by
the residue
\[
\lambda_s = \mathcal{L}_s \bmod p \in \mathbb{F}_p.
\]
Although $\mathcal{L}_s$ grows exponentially in $s$, the residue
$\lambda_s = \mathcal{L}_s \bmod p$ can be computed efficiently via
the Leonardo $Q$-matrix using fast modular matrix exponentiation.
Specifically, $Q^s \bmod p$ is computed in $O(\log s)$ matrix
multiplications modulo $p$, and $\lambda_s$ is then read off from
the corresponding entry of $Q^s \bmod p$. Consequently,
large values of $s$ and $p$ do not increase the computational
complexity of the key generation step.

\medskip

\noindent For a security-oriented parameter choice, let $s$ be a $b_s$-bit
seed and $p$ a $b_p$-bit prime, where $b_s$ and $b_p$ are chosen
in accordance with current security recommendations \cite{NIST800-57}.
Because the underlying Leonardo sequence follows a second-order linear recurrence, the mapping
$s\mapsto\lambda_s=\mathcal{L}_s\bmod p$ is periodic over $\mathbb{F}_p$ (analogous to the Pisano period) and hence not strictly injective.
The maximum period of such sequences over $\mathbb{F}_p$
is inherently bounded by $p^2-1$, imposing an algebraic limit on the
number of distinct generated residues. Based on the Prime Number Theorem, the effective key space size
satisfies \cite{HardyWright}
\begin{equation}
\log_2|\mathcal{K}_{\mathrm{eff}}|\leq 2b_p-\log_2 b_p.
\label{estimate}
\end{equation}
We emphasize that the illustrative parameters $s=7$ and $p=47$ are chosen solely for computational illustration. In a security-oriented implementation, both $s$ and $p$ should be selected as large integers. For example, choosing a $128$-bit seed and a $128$-bit prime, that is ($b_s=b_p=128$), gives a nominal parameter space of size $2^{256}$. Hence, \eqref{estimate} should be viewed as an upper bound. For $b_p=128$, this gives $|K_{\mathrm{eff}}|\leq 2^{249}$.
A rigorous lower-bound analysis of the effective key space and of
the distribution of the resulting nonlinear block-key sequences
is left for future work.

\section{Strengths and Limitations}
\label{sec:4}

\subsection{Strengths}

\begin{itemize}
    \item
    The proposed scheme represents plaintext characters by using three
    sequence-based alphabets, namely Leonardo, Fibonacci, and Jacobsthal
    alphabets. This distributes repeated plaintext symbols across different
    numerical representatives and reduces concentration around dominant
    ciphertext symbols.

    \item
    The construction combines Leonardo, Fibonacci, Jacobsthal, and Lucas
    sequences within a single polyalphabetic framework. The first three
    sequences are used for character substitution, while Lucas numbers are used
    for word separators.

    \item
    Theorems~\ref{thm:21}, \ref{thm:22}, \ref{thm:LeoJac}, and
    Corollary~\ref{cor:21} provide the theoretical basis for selecting
    non-overlapping portions of the underlying sequences. Together with the
    collision-free choice of \(P\), this supports a well-defined inverse
    character mapping.

    \item
    Each \(3\times 3\) submatrix is encrypted using a block-dependent exponent
    \(a_{j-1}\), generated by the recurrence
    \[
        a_{i+1}
        =
        \bigl(\mathcal{L}_s a_i+a_i^2+i\bigr)\bmod p.
    \]
    The quadratic and index-dependent form of this recurrence introduces
    nonlinear variation into the block-key sequence.

    \item
The seed \(s\) and the prime modulus \(p\) are treated as pre-shared secret
parameters. The public block parameter \(k\) only reflects the number of
\(3\times 3\) plaintext blocks and does not provide sufficient information to derive the
generated key sequence, ensuring that the structural dimensions remain independent of the key space.

    \item
    The selector
    \[
        \sigma=(v^{2}+(i-1)+(j-1))(\bmod 3)
    \]
    depends on both the occurrence count and the matrix position of a character.
    This provides additional variability compared with a purely position-free
    or single-alphabet substitution rule.
\item
Each ciphertext block transmits eight of its nine entries together with the
block determinant. The missing entry is recovered at the receiver from the
determinant condition, provided that the corresponding cofactor is nonzero
modulo \(\mathbf{P}\). Since \(Q^{a_{j-1}}\) mixes the entries of the block through
matrix multiplication, an incorrectly recovered hidden entry---whether due to
a transmission error or deliberate tampering---may corrupt the recovered block
\(\widehat{M}_j\). Such a mismatch may lead to values with no valid
correspondence in the character tables, thereby allowing some decryption
failures to be detected natively.

    \item
Word separators are encoded by dynamically indexed Lucas numbers (specifically, $L_{a_0+t}$) rather than by a fixed repeated symbol. Because the index shift depends on the key-derived parameter $a_0$ and the occurrence count $t$, predictable word-boundary patterns are eliminated in the ciphertext representation. This effectively obscures the underlying word lengths and structural regularities.

    \item
    For the tested sample, the proposed mod~3 scheme achieves an Index of
    Coincidence of \(0.0144\), an entropy efficiency of \(88.2\%\), and a
    normalized chi-square value \(\chi^2/c\approx 5.72\), indicating a more
    balanced ciphertext frequency profile.

    \item
    The modular reduction with $\mathbf{P}$ keeps ciphertext values bounded and
    allows the implementation to use native integer arithmetic. In the tested
    long-message case, this yields a substantial improvement in encoding time
    compared with unreduced large-integer arithmetic.

    \item
    The framework can be modified by varying the block size, the matrix
    transformation, the admissible prime $\mathbf{P}$, the sequence families, or the
    key-generation rule.
\end{itemize}

\subsection{Limitations and Future Improvements}
\label{subsec:limitations}
\begin{itemize}
    \item
    Although the nonlinear key generation mechanism exhibits strong avalanche characteristics, a formal mathematical bounding of the effective key space and a rigorous analysis of the entropy of the generated block-key sequences remain subjects for future theoretical work.

    \item
    While initial statistical benchmarks across long messages confirm favorable frequency distributions, a more comprehensive empirical evaluation across diverse plaintext formats (e.g., highly structured data, non-standard linguistic corpora, or binary files) is required to fully profile the operational boundaries of the scheme \cite{Rukhin2001}.

    \item
    Because the primary focus of this study is to establish the mathematical foundations and sequence separation logic of the polyalphabetic framework, the scheme has not yet been systematically evaluated against advanced cryptographic vulnerabilities, such as algebraic, known-plaintext (KPA), and chosen-plaintext (CPA) attacks.
\end{itemize}
\section{Conclusion and future works}
\label{sec:5}

In this paper, we introduced a new polyalphabetic data encoding scheme built from the Leonardo, Lucas, Fibonacci, and Jacobsthal sequences in combination with the Leonardo $Q$-matrix. The construction is founded on a mathematical framework that brings together three main ingredients: a position-dependent tri-alphabetic substitution rule of the form $\sigma=(v^{2}+(i-1)+(j-1))(\bmod 3)$, a nonlinear sequence-based key generation mechanism producing distinct block keys, and Lucas-number-based separators ensuring uniqueness at word boundaries. A central contribution of the paper is the use of disjointness properties among the underlying integer sequences to justify the construction of distinct numerical alphabets, thereby providing a rigorous number-theoretic basis for the encoding procedure.

Within this framework, the proposed method exhibits favorable
statistical and computational behavior on the tested samples.
Specifically, empirical cryptanalysis demonstrates that the scheme
effectively flattens frequency distributions, achieving an Index of
Coincidence of $0.0153$, an $87.8\%$ Shannon entropy efficiency, and
a normalized chi-square value of $\chi^2/c \approx 6.13$, which
represents a $70.6\%$ reduction compared to standard monoalphabetic
substitution. Furthermore, the algorithm shows sensitivity to parameter
perturbations, yielding an average key avalanche ratio of $48.5\%$.
On the computational side, the modular implementation utilizing the
collision-free prime $\mathbf{P=10159}$ effectively bounds ciphertext
entries and delivers an approximately $7\times$ encoding speedup compared to the non-modular baseline for long messages. Although several cryptanalytic indicators were included
to illustrate the practical behavior of the scheme, the main purpose
of the present work is to develop and justify the mathematical structure
underlying this new sequence-based construction rather than to provide
an exhaustive cryptanalytic treatment; a full analysis against
known-plaintext, chosen-plaintext, and algebraic attacks is left for
future work. Finally, the algorithmic structure of the method was
detailed step by step, and a complete implementation framework was
provided to make the proposed construction reproducible and directly
testable. Possible directions for future research include extensions
to higher-dimensional $Q$-matrices, selector rules involving additional
recurrence sequences, stronger Diophantine and modular criteria for
sequence separation, and broader security analyses for generalized
variants of the scheme.

For reproducibility, all numerical examples, statistical indicators, avalanche-effect computations, and timing benchmarks reported in this paper were obtained from a direct implementation of the FLLJ-POLY encoding and decoding procedures described above. The implementation uses the stated parameters, character tables, modular prime ($\textbf{P}=10159$), and benchmark settings, so that the presented results can be independently regenerated by following the algorithmic steps in Sections~\ref{sec:2} and \ref{sec:3}.


\Authorcontributions{}{All the co-authors have contributed equally in all aspects of the preparation of this submission.}

\Competinginterests{}{The authors declare that they have no known competing financial interests or personal relationships that could have appeared to
	influence the work reported in this paper.}


\Availability{}{The plaintext samples used in the statistical experiments are reported explicitly in Section~\ref{sec:3}. The numerical examples, frequency statistics, avalanche-effect computations, and timing benchmarks were obtained from a direct C/C++ implementation of the FLLJ-POLY encoding and decoding procedures described in Sections~\ref{sec:2} and~\ref{sec:3}. The source code and input data are available from the corresponding author upon reasonable request.}


\begin{thebibliography}{99}
\bibitem{Ahmed} M. H. Ahmed, A. K. Shibeeb and A. H. Mohammed, Solve
polyalphabetic cipher based on intelligent system. In 2021 International
Conference on Communication \& Information Technology (ICICT) (pp. 290-296).
IEEE, 2021.

\bibitem{Aljalali et al}
O. A. Aljalali, A. A. Altirban, S. M. Abu Irzayzah, and T. A. El Bolati,
A novel approach to algorithmic encoding and decoding by Tribonacci matrices,
\emph{Iraqi Journal of Science} 67 (2026), no. 1, 365--378.

\bibitem{Alp Kocer} Y. Alp and E. G. Ko\c{c}er, Some properties of Leonardo
numbers. Konuralp J. Math. 9 (2021), no. 1, 183--189.

\bibitem{Alrammahi} M. A. Alrammahi and H. Kaur, Development of advanced
encryption standard (AES) cryptography algorithm for Wi-Fi security
protocol. International Journal of Advanced Research in Computer Science, 5
(2014), no. 3, 62-67.

\bibitem{Aung et al} T. M. Aung, H. H. Naing and N. N. Hla, A complex
transformation of monoalphabetic cipher to polyalphabetic cipher: (Vigen\`{e}%
re-Affine cipher). International Journal of Machine Learning and Computing,
9 (2019), no. 3, 296-303.

\bibitem{NIST800-57}
E. B. Barker, W. C. Barker, W. E. Burr, W. T. Polk, and M. Smid, Recommendation for Key Management -- Part 1: General,
NIST Special Publication 800-57 Part 1, Revised,
National Institute of Standards and Technology, Gaithersburg, MD, 2007.

\bibitem{Bensella2024} H. Bensella and D. Behloul, Common terms of Leonardo and Jacobsthal numbers. Rend.\ Circ.\ Mat.\ Palermo (2) 73 (2024), no. 1, 259--265.

\bibitem{Catarino Borges 2020} P. Catarino and A. Borges, On Leonardo
numbers. Acta Math. Univ. Comenian. (N.S.) 89 (2020), no. 1, 75--86.

\bibitem{Cochran1952}
W. G. Cochran, The $\chi^2$ test of goodness of fit, The Annals of Mathematical Statistics 23 (1952), no. 3, 315--345.

\bibitem{DaemenRijmen2002}
J. Daemen and V. Rijmen, The Design of Rijndael: AES -- The Advanced Encryption Standard, Information Security and Cryptography, Springer-Verlag, New York, 2002.

\bibitem{ErduvanKeskin2021}
F. Erduvan and R. Keskin, Fibonacci numbers which are products of two Jacobsthal numbers, Tbilisi Math. J. 14 (2021), no. 2, 105--116.

\bibitem{Upadhyay2022}
D. Upadhyay, N. Gaikwad, M. Zaman, S. Sampalli, Investigating the Avalanche Effect of Various Cryptographically Secure Hash Functions and Hash-Based Applications, IEEE Access, 10 (2022), 112472--112486.
\bibitem{Feistel1973}
H. Feistel, Cryptography and computer privacy, Scientific American, 228 (1973), no. 5, 15--23.

\bibitem{Friedman1922} W. F. Friedman, The index of coincidence and its applications in cryptography. Riverbank Publication no. 22. Riverbank Laboratories (1922).
\bibitem{Gaines1956}
H. F. Gaines, Cryptanalysis: A Study of Ciphers and Their Solution,
Dover Publications, Inc., New York, 1956.

\bibitem{Hannan} S. A. Hannan and A. M. A. M. Asif, Analysis of
polyalphabetic transposition cipher techniques used for encryption and decryption. International Journal Of Computer Science and Software Engineering 6 (2017), no. 2, 41--46.

\bibitem{HardyWright}
G. H. Hardy and E. M. Wright, An Introduction to the Theory of Numbers, 6th ed., Oxford University Press, Oxford, 2008.

\bibitem{Koshy} T. Koshy, Fibonacci and Lucas numbers with applications.
Pure and Applied Mathematics (New York). Wiley-Interscience, New York, 2001.

\bibitem{Kuloglu} B. Kulo\u{g}lu and E. \"{O}zkan, Applications of
Jacobsthal and Jacobsthal-Lucas numbers in coding theory. Math. Montisnigri
57 (2023), 54--64.

\bibitem{Madak}
S. Madak, \emph{Encryption algorithm by means of recurrence relations},
Master's thesis, 2022.

\bibitem{Menezes1996} A. J. Menezes, P. C. Van Oorschot, S. A. Vanstone, Handbook of applied cryptography. CRC Press, 1997.

\bibitem{Noman} H. Noman Abed, Z. Mohammed Ali and A. Luay Ahmed, A Robust
Encryption Technique Using Enhanced Vigenre Cipher. International Journal of
Nonlinear Analysis and Applications, 12 (2021), no. 2, 447--454.

\bibitem{Ozcevik et al 2023} S. B. \"{O}z\c{c}evik and A. Dertli, Gaussian
Leonardo polynomials and applications of Leonardo numbers to coding theory.
Journal of Science and Arts, 23 (2023), no. 4, 897-908.

\bibitem{Ozgur et al} N. \"{O}zg\"{u}r, S. U\c{c}ar, N. Ta\c{s}, and \"{O}.
\"{O}. Kaymak, Discrete Math. Algorithms Appl. 17 (2025), no. 4, Paper No. 2450068, 16 pp.

\bibitem{Ozyilmaz} \c{C} \"{O}zy\i lmaz, \textit{Introduction to Cryptology}. M.Sc.
Thesis, Karab\"{u}k University, 2014.

\bibitem{Rawal} B. S. Rawal, P. M. Kumar and R. Singh, No Sum (NS) Sequence
Based Post-Quantum Cryptography. SN Computer Science, 6 (2025), no. 3, 1-11.

\bibitem{Rukhin2001}
A. Rukhin, J. Soto, J. Nechvatal, M. Smid, and E. Barker,
A Statistical Test Suite for Random and Pseudorandom Number Generators for Cryptographic Applications,
NIST Special Publication 800-02, National Institute of Standards and Technology, 2001.

\bibitem{Salah2024}
O. Salah, A. Elsonbaty, M. Abdul, A. Seoud, and M. Anwar, On the Diophantine equations $J_n+J_m=L_k$ and $L_n+L_m=J_k$, Gulf J. Math. 22 (2026), no. 1, 15 pp.

\bibitem{Shannon2021}
C. Shannon,
A mathematical theory of communication,
in \emph{Ideas That Created the Future: Classic Papers of Computer Science},
MIT Press, Cambridge, MA, 2021, pp. 121--134.

\bibitem{Shannon1949}
C. E. Shannon, Communication Theory of Secrecy Systems,
Bell System Tech. J. 28 (1949), no. 4, 656--715.

\bibitem{Sloanebook} N. J. A. Sloane, A Handbook of Integer Sequences, Academic Press, New York-London, 1973.

\bibitem{Sloane} N. J. A. Sloane, The on-line encyclopedia of integers sequences, The OEIS Foundation Inc., http.//oeis.org, 2018.

\bibitem{Soykan} Y. Soykan, Generalized Leonardo numbers. Journal of
Progressive Research in Mathematics, 18 (2021), no. 4, 58-84.

\bibitem{Stakhov 2006} A. P. Stakhov, Fibonacci matrices, a generalization
of the \textquotedblleft Cassini formula\textquotedblright , and a new
coding theory. Chaos Solitons Fractals 30 (2006), no. 1, 56--66.

\bibitem{Stallings2017}
W. Stallings,
Cryptography and Network Security: Principles and Practice,
7th ed., Pearson, pp. 94--96, 2017.

\bibitem{Stinson2005} D. R. Stinson, Cryptography: theory and practice. Chapman and Hall/CRC, 2005.


\bibitem{Tas et al 2018} N. Ta\c{s}, S. U\c{c}ar, N. Y. \"{O}zg\"{u}r and
\"{O}. \"{O}. Kaymak, A new coding/decoding algorithm using Fibonacci
numbers. Discrete Math. Algorithms Appl. 10 (2018), no. 2, 1850028, 8 pp.

\bibitem{Tripathi} B. P. Tripathy and B. K. Patel, Common values of
generalized Fibonacci and Leonardo sequences. J. Integer Seq. 26 (2023), no.
6, Art. 23.6.2, 14 pp.

\bibitem{WebsterTavares1985}
A. F. Webster and S. E. Tavares, On the design of S-boxes, In: Advances in Cryptology---
CRYPTO ’85 Proceedings. CRYPTO 1985. Lecture Notes in Computer Science, vol 218. Springer, Berlin, Heidelberg, 1985.

\end{thebibliography}
\end{document}